\documentclass[11pt,reqno]{amsart}
\usepackage{ytableau,tikz,hyperref}
\usepackage{palatino}
\usepackage{mathtools}
\usepackage[left=3cm,right=3cm,top=3cm,bottom=3cm]{geometry}
\usepackage{mathrsfs}
\usepackage[shortlabels,inline]{enumitem}

\newcommand\Par{\mathscr{P}}

\newtheorem{theorem}{Theorem}
\newtheorem{remark}{Remark}
\newtheorem{example}{Example}
\newcommand\bremark{\begin{remark}\begin{upshape}}
\newcommand\eremark{\end{upshape}\end{remark}}
\newtheorem{proposition}{Proposition}
\newtheorem{corollary}{Corollary}

\DeclareMathOperator{\VHive}{VHive}
\DeclareMathOperator{\EHive}{EHive}
\DeclareMathOperator{\SSYT}{SSYT}

\DeclareMathOperator{\GT}{GT}
\DeclareMathOperator{\NE}{NE}
\DeclareMathOperator{\SE}{SE}
\DeclareMathOperator{\VE}{VE}
\DeclareMathOperator{\U}{\mathfrak{U}}
\DeclareMathOperator{\gl}{\mathfrak{gl}}

\allowdisplaybreaks
\title[A second reduction-type formula for the refined LR coefficients in type A]{A second reduction-type formula for the refined Littlewood--Richardson coefficients in type A}
\author[]{Siddheswar Kundu}
\address{Department of Mathematics, Indian Institute of Technology Kanpur, Kanpur 208016, India.}
\email{kundusidhu96@gmail.com}
\begin{document}
\begin{abstract}
For a permutation $w$ in the symmetric group $S_n$ and partitions $\lambda, \mu, \nu $ with at most $n$ parts,
the refined Littlewood--Richardson (LR) coefficients $c^{\nu}_{\lambda,\mu}(w)$ in type $A_n$ count the multiplicity of the irreducible polynomial representation $V(\nu)$ of the general linear algebra $\gl_n(\mathbb{C})$ appearing in the decomposition of the Kostant--Kumar submodule $K(\lambda,w,\mu)$ of the tensor product $V(\lambda) \otimes V(\mu)$ of two irreducible polynomial $\gl_n(\mathbb{C})$-modules. In this paper, we establish a second reduction-type formula for $c^{\nu}_{\lambda,\mu}(w)$, extending the second reduction formula for the classical Littlewood--Richardson coefficients $c^{\nu}_{\lambda,\mu}$. The proof relies on the hive model.     
\end{abstract}
\maketitle
\section{Introduction}
The Littlewood--Richardson coefficients, represented by $c^{\nu}_{\lambda,\mu}$, lie at the intersection of three major mathematical fields: algebraic combinatorics, representation theory, and algebraic geometry. Originally introduced to describe the multiplication of Schur functions $(s_{\lambda}s_{\nu}=\sum_{\nu}c^{\nu}_{\lambda,\mu}s_{\nu})$, these non-negative integers fundamentally govern the decomposition of tensor products of irreducible polynomial representations for the general linear algebra $\gl_n(\mathbb{C})$ and the cup product of two Schubert classes in the cohomology ring of a Grassmannian.

Let $V(\lambda)$ be the irreducible polynomial representaion of $\mathfrak{gl}_n(\mathbb{C})$ corresponding to the dominant integral weight $\lambda$ and $v_{\lambda}$ be a non-zero highest weight vector in $V(\lambda)$. For an element $w$ of the Weyl group $S_n$, let $v_{w\mu}$ denote a non-zero vector in the one-dimensional weight space of weight $w\mu$ in $V(\mu)$. Then the cyclic $\gl_n(\mathbb{C})$-submodule $\U\gl_n(v_{\lambda} \otimes v_{w\mu})$ of $V(\lambda) \otimes V(\mu)$, where $\U\gl_n$ denotes the universal enveloping algebra of $\gl_n(\mathbb{C})$, is called the {\em Kostant--Kumar module} \cite[\S5]{Mrigendra:adv:arxiv}, and it is denoted by $K(\lambda, w, \mu)$. The decomposition rule of $K(\lambda, w, \mu)$ \cite[\S9.4]{Mrigendra:adv:arxiv} is expressed as follows:
$$ K(\lambda, w, \mu)= \bigoplus_{\nu} V(\nu)^{\oplus c^{\nu}_{\lambda,\mu}(w)}.$$
Following \cite[\S9.3.2]{Mrigendra:adv:arxiv}, we refer to $c^{\nu}_{\lambda,\mu}(w)$ as the {\em refined Littlewood--Richardson coefficients}. Since $K(\lambda, w_0, \mu) = V(\lambda) \otimes V(\mu)$ \cite[Remark 5.3]{Mrigendra:adv:arxiv}, where $w_0$ denotes the longest element in $S_n$, it follows that $c^{\nu}_{\lambda,\mu}(w_0)$ specializes to the classical coefficient $c^{\nu}_{\lambda,\mu}$.

In \cite{Hive-main}, two reductive formulae for the classical Littlewood--Richardson coeﬃcients are presented in geometric terms. Their combinatorial equivalents are as follows:
\begin{theorem}
\begin{upshape}
    (First reduction formula \cite[Theorem 2.1]{Hive-JCTA})
\end{upshape}
   Let $\lambda,\mu,\nu $ be partitions such that $\lambda,\mu \subseteq \nu$ and $|\lambda|+|\mu|=|\nu|$. We assume that $\nu$ has $k$ nonzero parts. Also, if for any three indices $\alpha,\beta,\gamma$ such that $\alpha +\beta + \gamma =2k+1$, then 
   \begin{equation*}
 c^{\nu}_{\lambda,\mu} = \left\{ 
    \begin{array}{ll}
     c^{\nu-\nu_{\alpha+\beta-k}}_{\lambda-\lambda_{\alpha},\mu-\mu_{\beta}}.
         & \quad \text{if} \quad \lambda_{\alpha}+\mu_{\beta}=\nu_{\alpha+\beta-k},  \\
        0 & \quad \text{if} \quad \lambda_{\alpha}+\mu_{\beta} > \nu_{\alpha+\beta-k}. 
    \end{array}  \right. 
\end{equation*}
\end{theorem}
\begin{theorem}\begin{upshape}
    (Second reduction formula \cite[Theorem 2.3]{Hive-reduction})
\end{upshape}
\label{thm:second}
   Consider three partitions, $\lambda,\mu,\nu $, each with at most $n$ parts such that $|\lambda|+|\mu|=|\nu|$. Assume there exist indices $i,j,k$ such that $i+j+k=n$ and $\lambda_i+\mu_j>\nu_1 +\nu_{n-k+1}$. Suppose, moreover, that $\lambda_i >\lambda_{i+1}, \mu_i >\mu_{i+1}$ and $\nu_{n-k}>\nu_{n-k+1}$. Then we obtain 
   $$ c^{\nu}_{\lambda,\mu} = c^{\nu-(1^{n-k})}_{\lambda-(1^i),\mu-(1^j)}.$$
\end{theorem}
Bijective proofs for the usual Littlewood--Richardson coefficient reduction formulae were established in \cite{Hive-JCTA,Hive-Korean}, followed by a hive-model proof for the second reduction formula in \cite{Hive-reduction}. In Theorem~\ref{thm:main}, we present a new second reduction-type formula for $c^{\nu}_{\lambda,\mu}(w)$, which we prove via the hive model.
\section{Tableau model}
Let $\mathbb{Z}_{\geq 0}=\{0,1,2,\dots \}$. For a positive integer $n$, we use $\Par_n$ to represent the set of partitions with at most $n$ non-zero parts, i.e.,
$$ 
\Par_n:= \{ \lambda=(\lambda_1,\lambda_2,\dots,\lambda_n) \in \mathbb{Z}^n_{\geq 0} :\lambda_1 \geq  \dots \geq \lambda_n \geq 0 \}.
$$
For $\lambda \in \Par_n$, its {\em size} is $|\lambda|:=\lambda_1 +\cdots + \lambda_n$ and its {\em length} $\ell(\lambda)$ is the number of non-zero parts.
Graphically, we identify $\lambda$ with its Young diagram $Y(\lambda)$, by placing $\lambda_i$ square cells in row $i$, aligning every row to the top and left. The figure below shows the Young diagram $Y(4,3,1)$.
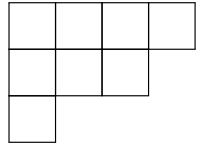
\begin{center}
 \ytableausetup{nosmalltableaux,boxsize=0.6 cm}
\begin{ytableau} 
\null & \null & \null & \null \\
\null & \null & \null \\
\null \\
\end{ytableau}  
\end{center}
A {\em semistandard Young tableau} of shape $\mu$ with {\em content} $\delta =(\delta_1,\delta_2,\dots)$ is a filling of boxes of the Young diagram $Y(\mu)$ with $\delta_j$ copies of the number $j$ such that the entries are:
\begin{itemize}
    \item weakly increasing along each row, and
    \item strictly increasing down each column.
\end{itemize}
The {\em reverse row word} of a semistandard Young tableau $T$, denoted by $r_T$, is obtained by reading the entries of $T$ row by row from right to left, moving from the top row to the bottom row (as illustrated below).
$$
T = \ytableaushort{2334,345,4} \ \text{and} \ r_T = 43325434
$$
For a partition $\mu \in \Par_n$, let $\SSYT(\mu)$ denote the set of semistandard Young tableaux of shape $\mu$ with entries in $\{1, 2, \dots, n\}$. The crystal raising and lowering operators $e_i, f_i $ $(1 \leq i \leq n-1)$ act on the full set of words in the alphabet $\{1, 2, \dots, n\}$ \cite[Chapter 5]{Lothaire}, and thus act on the set $\SSYT(\mu)$ through their associated reverse row words.

Given $w \in S_n$ with a fix reduced decomposition $w = s_{i_1} s_{i_2} \cdots s_{i_k}$ and a partition $\mu \in \Par_n$, the {\em Demazure crystal} $\mathfrak{B}_w(\mu)$ is given by:
$$  \mathfrak{B}_w(\mu) :=\{f_{i_1}^{t_1} f_{i_2}^{t_2} \cdots f_{i_k}^{t_k} r_{T^{\mu}}: t_j \geq 0\} ,$$
where $T^{\mu}$ is the unique semistandard Young tableau whose shape and weight are both given by $\mu$.
For $w=w_0,$ we obtain $\mathfrak{B}_{w_0}(\mu) = \SSYT(\mu)$.

A word $w = w_1w_2\cdots w_k$ in the alphabet $\{1,2,\dots,n\}$ is called {\em dominant} if for every $1 \leq i <n$, any prefix $ w_1w_2\cdots w_t$ (for $1 \leq t \leq k$) of $w$ contains at least as many occurrences of $i$ as $i+1$. For example, $12312$ is dominant, but $12231$ is not.
\begin{theorem}
\begin{upshape}
    \cite[Theorem 5.25]{Joseph} \cite[Theorem 4.1]{Mrigendra:saturation}
\end{upshape}
\label{thm:joseph}
For $\lambda,\mu,\nu \in \Par_n,$ and $w \in S_n,$ let
\[ \SSYT_{\lambda,\mu}^{\nu}(w) :=\{T \in \mathfrak{B}_w(\mu): r_{T^{\lambda}}*r_T \text{ is a dominant word of weight } \nu\},\] where $*$ denotes concatenation of words. Then $c^{\nu}_{\lambda,\mu}(w)$ is the cardinality of the set $\SSYT_{\lambda,\mu}^{\nu}(w)$.
\end{theorem}
\begin{corollary}
\label{cor:LR-rule}
Since $\mathfrak{B}_{w_0}(\mu) = \SSYT(\mu)$, $c^{\nu}_{\lambda,\mu}$ is the cardinality of the set \[ \SSYT_{\lambda,\mu}^{\nu}:= \SSYT_{\lambda,\mu}^{\nu}(w_0) =\{T \in \SSYT(\mu): r_{T^{\lambda}}*r_T \text{ is a dominant word of weight } \nu\}.\]
\end{corollary}
\section{Gelfand--Tsetlin pattern and Kogan faces}
\label{Sec:3}
A Gelfand--Tsetlin (GT) pattern of size $n$ is a triangular array  of real numbers, denoted by $X=(x_{i,j})_{1\leq j \leq i \leq n} \in \mathbb{R}^{n(n+1)/2}$ (see Figure~\ref{figure:array}), that satisfies the following ``North-East" (NE) and ``South-East" (SE) inequalities for $ 1 \leq j < i \leq n$:
$$
\NE_{i,j}(X)= x_{i,j} - x_{i-1,j} \geq 0 \text{ and }
\SE_{i,j}(X)=  x_{i-1,j} - x_{i,j+1} \geq 0.
$$
For $\lambda \in \Par_n,$ we let $\GT(\lambda)$ denote the set of all GT patterns $X=(x_{i,j})_{1\leq j \leq i \leq n} \in \mathbb{R}^{n(n+1)/2} $ whose bottom row matches $\lambda$, that is,
$$x_{n,j}=\lambda_{j} \text{ for } 1 \leq j \leq n .$$
Additionally, we define the subset of integer-labeled GT patterns in $\GT(\lambda)$ by $$\GT_{\mathbb{Z}}(\lambda):= \GT(\lambda) \cap \mathbb{Z}^{n(n+1)/2}.$$
\begin{figure}
    \centering
    \begin{tikzpicture}[scale=1.2]
    \draw (2.8,2*1.732) node {$x_{1,1}$};
	\draw (2.1,1.5*1.732) node {$x_{2,1}$};
    \draw (3.5,1.5*1.732) node {$x_{2,2}$};
	\draw (1.4,1.732) node {$x_{3,1}$};
	\draw (2.8,1.732) node {$x_{3,2}$};
    \draw (4.2,1.732) node {$x_{3,3}$};
	\draw (0.7,0.5*1.732) node {$x_{4,1}$};
	\draw (2.1,0.5*1.732) node {$x_{4,2}$};
	\draw (3.5,0.5*1.732) node {$x_{4,3}$};
    \draw (4.9,0.5*1.732) node {$x_{4,4}$};
	\draw (0,0) node {$x_{5,1}$};
	\draw (1.4,0) node {$x_{5,2}$};
	\draw (2.8,0) node {$x_{5,3}$};
	\draw (4.2,0) node {$x_{5,4}$};
    \draw (5.6,0) node {$x_{5,5}$};
\end{tikzpicture}
    \caption{A Gelfand--Tsetlin array for $n=5$}
    \label{figure:array}
\end{figure}
\begin{example}
\label{example:GT}
An example of a GT pattern in $\GT_{\mathbb{Z}}(4,3,2,0)$ is given below:
$$
\begin{tikzpicture}[scale=1]
	\draw (1.5,1.5*1.732) node {$2$};
	\draw (1,1.732) node {$2$};
	\draw (2,1.732) node {$1$};
	\draw (0.5,0.5*1.732) node {$4$};
	\draw (1.5,0.5*1.732) node {$2$};
	\draw (2.5,0.5*1.732) node {$1$};
	\draw (0,0) node {$4$};
	\draw (1,0) node {$3$};
	\draw (2,0) node {$2$};
	\draw (3,0) node {$0$};
    \draw (-0.5,1.5) node {$X=$};
\end{tikzpicture}.
$$
\end{example}
Let $X \in \GT_{\mathbb{Z}}(\lambda)$. Then for $1 \leq i \leq n,$ the row $x^{(i)}:=(x_{i,1},x_{i,2},\ldots, x_{i,i})$ defines a partition with $\ell(x^{(i)}) \leq i$. The definition of a GT pattern ensures that each skew shape $x^{(i)}/x^{(i-1)} (\text{taking }x^{(0)}: = \emptyset)$ is a horizontal strip, i.e., it does not contain a vertical domino. In fact, $X \in \GT_{\mathbb{Z}}(\lambda)$ if and only if $x^{(i)}/x^{(i-1)}$ is a horizontal strip for all $i$.

Filling the boxes of $x^{(i)}/x^{(i-1)}$ with the label $i$ for each $1 \leq i \leq n$ produces a semistandard Young tableau $\tilde{X}$. The resulting assignment defines a well-known bijection:
$$\Upsilon: \GT_{\mathbb{Z}}(\lambda) \rightarrow \SSYT_n(\lambda), X \mapsto \tilde{X}.$$
For instance, for the GT pattern in Example~\ref{example:GT}, we have
$$X^{(1)}=  \ytableausetup{mathmode,
notabloids,boxsize=1.5em}
  \begin{ytableau}
    1&1  
  \end{ytableau} \quad
  X^{(2)}=  \ytableausetup{mathmode,
notabloids,boxsize=1.5em}
  \begin{ytableau}
    1&1\\2   
  \end{ytableau} \quad 
  X^{(3)}=  \ytableausetup{mathmode,
notabloids,boxsize=1.5em}
  \begin{ytableau}
    1&1&3&3\\2&3\\3   
  \end{ytableau} \quad
 \Tilde{X}=X^{(4)}=  \ytableausetup{mathmode,
notabloids,boxsize=1.5em}
  \begin{ytableau}
    1&1&3&3\\2&3&4\\3&4   
  \end{ytableau}.
$$
\subsection{Kogan faces}
\label{sec:Kogan}
Let $F$ be a subset of $\{(i,j): 1 \leq j < i \leq n \}$. Then the {\em Kogan face} (see \cite[\S2.2.1]{Kogan},\cite[\S3.3 \& \S4.3]{Fujita:2}, \cite[\S5]{Fujita}) of $\GT(\lambda)$ associated to the face $F$ is defined by
$$\GT(\lambda,F):=\{X \in \GT(\lambda): \NE_{i,j}(X)=0 \text{ for } (i,j) \in F \}.$$
Next, for each pair $(i,j)$ satisfying $1 \leq j < i \leq n,$ we associate the simple transposition $s_{i-j}$. This construction is illustrated below for $n=5$.
$$ 
\begin{tikzpicture}[scale=1.2]
    \draw (2.8,2*1.732) node {$\bullet$};
	\draw (2.1,1.5*1.732) node {$\bullet$};
    \draw (3.5,1.5*1.732) node {$\bullet$};
	\draw (1.4,1.732) node {$\bullet$};
	\draw (2.8,1.732) node {$\bullet$};
    \draw (4.2,1.732) node {$\bullet$};
	\draw (0.7,0.5*1.732) node {$\bullet$};
	\draw (2.1,0.5*1.732) node {$\bullet$};
	\draw (3.5,0.5*1.732) node {$\bullet$};
    \draw (4.9,0.5*1.732) node {$\bullet$};
	\draw (0,0) node {$\bullet$};
	\draw (1.4,0) node {$\bullet$};
	\draw (2.8,0) node {$\bullet$};
	\draw (4.2,0) node {$\bullet$};
    \draw (5.6,0) node {$\bullet$};
    \draw[blue] (0.15,0.15) -- (0.7-0.15,0.5*1.732-0.15);
    \draw[blue] (0.7+0.15,0.5*1.732+0.15) -- (1.4-0.15,1.732-0.15);
    \draw[blue] (1.4+0.15,1.732+0.15) -- (2.1-0.15,1.5*1.732-0.15);
    \draw[blue] (2.1+0.15,1.5*1.732+0.15) -- (2.8-0.15,2*1.732-0.15);
    \draw[blue] (1.4+0.15,0+0.15) -- (2.1-0.15,0.5*1.732-0.15);
    \draw[blue] (2.1+0.15,0.5*1.732+0.15) -- (2.8-0.15,1.732-0.15);
    \draw[blue] (2.8+0.15,1.732+0.15)--(3.5-0.15,1.5*1.732-0.15);
    \draw[blue] (2.8+0.15,0+0.15) -- (3.5-0.15,0.5*1.732-0.15);
    \draw[blue] (3.5+0.15,0.5*1.732+0.15) -- (4.2-0.15,1.732-0.15);
    \draw[blue] (4.2+0.15,0+0.15)--(4.9-0.15,0.5*1.732-0.15);
    \draw[blue] (0.7/2-0.2,0.5*1.732/2+0.1) node {$s_4$};
    \draw[blue] (3.5/2-0.2,1.732/4+0.1) node {$s_3$};
    \draw[blue] (6.3/2-0.2,1.732/4+0.1) node {$s_2$};
    \draw[blue] (9.1/2-0.2,1.732/4+0.1) node {$s_1$};
    \draw[blue] (2.1/2-0.2,0.75*1.732+0.1) node {$s_3$};
    \draw[blue] (2.1/2-0.2+1.4,0.75*1.732+0.1) node {$s_2$};
    \draw[blue] (2.1/2-0.2+2.8,0.75*1.732+0.1) node {$s_1$};
    \draw[blue] (3.5/2-0.2,2.5*1.732/2+0.1) node {$s_2$};
    \draw[blue] (3.5/2-0.2 +1.4,2.5*1.732/2+0.1) node {$s_1$};
    \draw[blue] (4.9/2-0.2,1.75*1.732+0.1) node {$s_1$};
\end{tikzpicture}
$$
We order the elements of $F$ in lexicographically, where $(i, j) \prec (i', j')$ if and only if:
\begin{itemize}
    \item either $i < i',$
    \item or $i = i'$ and $j < j'$.
\end{itemize}
By writing the corresponding factors $s_{i-j}$ from left to right in this order, we define their product in $S_n$ to be $w_F$. When the length of $w_F = |F|,$ the subset $F$ is called {\em reduced}. Following \cite[Definition 5.1]{Fujita}, we then set
$$\hat{w}_F=w_0w_Fw_0,$$ where $w_0$ is the longest element in $S_n$. To illustrate, consider $F_1=\{ (2,1),(3,1),(5,2) \}$ and $F_2=\{(3,2),(4,3)\}$. Then $w_{F_1}=s_1s_2s_3$ is reduced, but $w_{F_2} = s_1s_1$ is not. 

For any $w \in S_n, \lambda \in \Par_n,$ we define
$$\GT(\lambda,w):= \bigcup_{ \substack{F: F \text{ is reduced},\\ \hat{w}_F=w} } GT(\lambda,F).$$
To illustrate this definition, let $n=5, w=s_3s_2,$ so that $w_0s_3s_2w_0=s_2s_3$. Then we have 
$$\GT(\lambda,s_3s_2) = \GT(\lambda,\{(3,1),(4,1)\}) \bigcup \GT(\lambda,\{(3,1),(5,2)\}) \bigcup \GT(\lambda,\{(4,2),(5,2)\}).$$
\begin{proposition}
\begin{upshape}
   \cite[Corollary 5.19]{Fujita} 
\end{upshape}
For $\lambda \in \Par_n$, $ w \in S_n$, the bijection $\Upsilon: \GT_{\mathbb{Z}}(\lambda) \rightarrow \SSYT_n(\lambda)$ maps $\GT_{\mathbb{Z}}(\lambda,w_0w) $ bijectively onto $ \mathfrak{B}_{w}(\lambda)$. 
\end{proposition}
\section{Hive}
\label{sec:hive}
An {\em array} $h$ of order $n$ is a triangular arrangement of real numbers, denoted by $(h_{i j})_{0 \leq j \leq i \leq n}$, where each number labels a specific vertex in an equilateral triangular graph. For $n = 4$ the arrangement is shown below:
$$
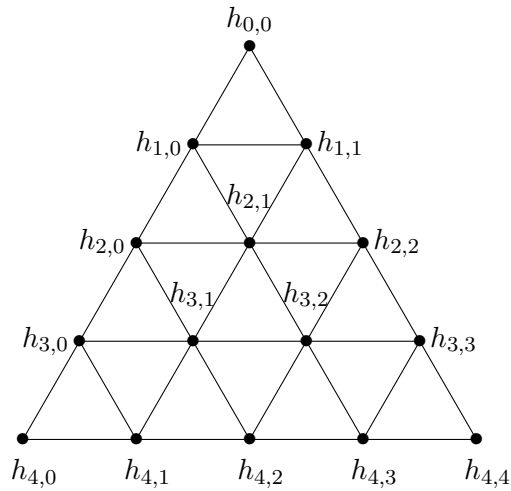
\begin{tikzpicture}[scale=1.5]
    \node at (0,0) {$\bullet$};
    \node at (1,0) {$\bullet$};
    \node at (2,0) {$\bullet$};
    \node at (3,0) {$\bullet$};
    \node at (4,0) {$\bullet$};
    \node at (0.5,0.5*1.732) {$\bullet$};
    \node at (1.5,0.5*1.732) {$\bullet$};
    \node at (2.5,0.5*1.732) {$\bullet$};
    \node at (3.5,0.5*1.732) {$\bullet$};
    \node at (1,1.732) {$\bullet$};
    \node at (2,1.732) {$\bullet$};
    \node at (3,1.732) {$\bullet$};
    \node at (1.5,1.5*1.732) {$\bullet$};
    \node at (2.5,1.5*1.732) {$\bullet$};
    \node at (2,2*1.732) {$\bullet$};
    \node at (2,2.13*1.732) {$h_{0,0}$};
    \node at (1.2,1.5*1.732) {$h_{1,0}$};
    \node at (2.8,1.5*1.732) {$h_{1,1}$};
    \node at (1-0.3,1.732) {$h_{2,0}$};
    \node at (2,1.732+0.4) {$h_{2,1}$};
    \node at (3+0.3,1.732) {$h_{2,2}$};
    \node at (0.2,0.5*1.732) {$h_{3,0}$};
    \node at (1.5,0.5*1.732+0.4) {$h_{3,1}$};
    \node at (2.5,0.5*1.732+0.4) {$h_{3,2}$};
    \node at (4-0.2,0.5*1.732) {$h_{3,3}$};
    \node at (0.1,-0.3) {$h_{4,0}$};
    \node at (1.1,-0.3) {$h_{4,1}$};
    \node at (2.1,-0.3) {$h_{4,2}$};
    \node at (3.1,-0.3) {$h_{4,3}$};
    \node at (4.1,-0.3) {$h_{4,4}$};
    \draw (0,0) -- (4,0);
    \draw (0.5, 0.5*1.732) -- (3.5, 0.5*1.732);
    \draw (1, 1.732) -- (2, 1.732) -- (3, 1.732);
    \draw (1.5, 1.5*1.732) -- (2.5, 1.5*1.732);
    \draw (0,0) -- (0.5, 0.5*1.732) -- (1, 1.732) -- (1.5, 1.5*1.732) -- (2, 2*1.732);
    \draw (1,0) -- (1.5, 0.5* 1.732) -- (2, 1.732) -- (2.5, 1.5*1.732);
    \draw (2,0) -- (2.5, 0.5* 1.732) -- (3, 1.732);
    \draw (3,0) -- (3.5, 0.5* 1.732);
    \draw (1,0) -- (0.5, 0.5*1.732);
    \draw (2,0) -- (1.5, 0.5*1.732) -- (1,1.732);
    \draw (3,0) -- (2.5, 0.5*1.732);
    \draw (2.5, 0.5*1.732) -- (2, 1.732) -- (1.5, 1.5*1.732);
    \draw (4,0) -- (3.5, 0.5*1.732);
    \draw (3.5, 0.5*1.732) -- (3, 1.732) -- (2.5, 1.5*1.732) -- (2, 2*1.732);
\end{tikzpicture}
$$
An array $h$ of order $n$ is said to be an {\em $n$-hive} if it satisfies certain hive conditions, which are constraints applied to the vertex labels of each elementary rhombus (formed by two adjacent triangles). These elementary rhombi come in three distinct orientations as shown below.
$$ 
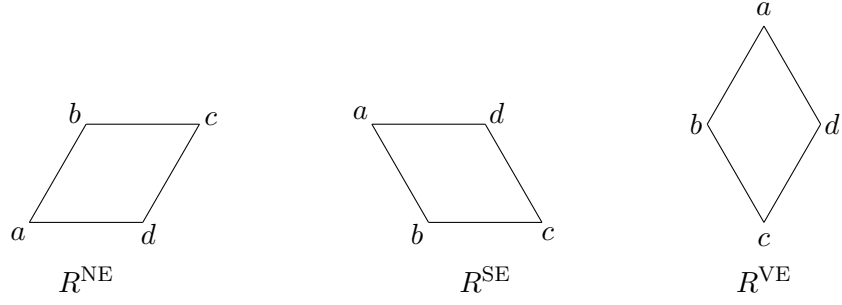
\begin{tikzpicture}[scale=1.5]
  \draw (0,0) -- (1,0);
  \draw (0,0) -- (0.5, 0.5*1.732);
  \draw (0.5, 0.5*1.732) -- (1.5, 0.5*1.732);
  \draw (1,0) -- (1.5,0.5*1.732);
  \node at (0.5,-0.5) {$R^{\NE}$};
  \node at (-0.1,-0.1) {$a$};
  \node at (0.5-0.1, 0.5*1.732+0.1) {$b$};
  \node at (1.5+0.1, 0.5*1.732+0.05) {$c$};
  \node at (1+0.05,0-0.1) {$d$};
\end{tikzpicture}   \quad \quad \quad \quad 
\begin{tikzpicture}[scale=1.5]
  \draw (0,0) -- (1,0);
  \draw (0,0) -- (-0.5, 0.5*1.732);
  \draw (-0.5, 0.5*1.732) -- (0.5, 0.5*1.732);
  \draw (1,0) -- (0.5,0.5*1.732);
  \node at (0.5,-0.5) {$R^{\SE}$};
  \node at (-0.1,-0.1) {$b$};
  \node at (-0.5-0.1, 0.5*1.732+0.1) {$a$};
  \node at (0.5+0.1, 0.5*1.732+0.1) {$d$};
  \node at (1+0.05,0-0.1) {$c$};
\end{tikzpicture}   \quad \quad \quad \quad 
\begin{tikzpicture}[scale=1.5]
  \draw (0,0) -- (0.5,0.5*1.732);
  \draw (1,0) -- (0.5,0.5*1.732);
  \draw (0.5,-0.5*1.732) -- (0,0);
  \draw (0.5,-0.5*1.732) -- (1,0);
  \node at (0.5,-0.5*1.732-0.5) {$R^{\VE}$};
  \node at (-0.1,0) {$b$};
  \node at (0.5,0.5*1.732+0.15) {$a$};
  \node at (1+0.1,0) {$d$};
  \node at (0.5,-0.5*1.732-0.15) {$c$};
\end{tikzpicture}.$$
In each case, the vertex labeling above yields the hive condition $b + d \geq a + c$. Furthermore, an $n$-hive is classified as an integer hive if all its entries are integers.

Fix partitions $\lambda,\mu,\nu \in \Par_n $ such that $|\lambda| + |\mu| = |\nu|$. We then define $\VHive(\lambda,\mu,\nu)$ by the set of all hives $h$ satisfying $h_{0,0}=0,$ along with the following boundary conditions for $1 \leq i \leq n$:
\begin{itemize}
    \item $h_{i,0}=\lambda_1 + \cdots +\lambda_i,$
    \item $h_{n,i}=|\lambda|+ \mu_1 + \cdots + \mu_i,$
    \item $h_{i,i}=\nu_1 + \cdots + \nu_i$.
\end{itemize}
\begin{theorem}
\begin{upshape}
    \cite{Buch:saturation, Mrigendra:saturation}
\end{upshape}
$c_{\lambda,\mu}^{\nu}$ is the cardinality of $\VHive_{\mathbb{Z}}(\lambda,\mu,\nu),$ where $$\VHive_{\mathbb{Z}}(\lambda,\mu,\nu):=\VHive(\lambda,\mu,\nu) \cap \mathbb{Z}^{(n+1)(n+2)/2}.$$      
\end{theorem}
\subsection{Hive Kogan face} 
\label{sec:hive-kogan}
For $1 \leq j < i \leq n,$ we consider the following rhombus 
$$
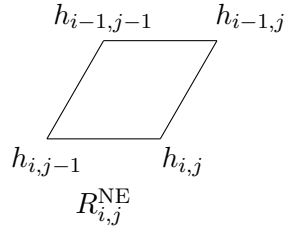
\begin{tikzpicture}[scale=1.5]
  \draw (0,0) -- (1,0);
  \draw (0,0) -- (0.5, 0.5*1.732);
  \draw (0.5, 0.5*1.732) -- (1.5, 0.5*1.732);
  \draw (1,0) -- (1.5,0.5*1.732);
  \draw (0.5, 0.5*1.732+0.2) node {$h_{i-1,j-1}$};
  \draw (1.5+0.3, 0.5*1.732+0.2) node {$h_{i-1,j}$};
  \draw (0,-0.2) node {$h_{i,j-1}$};
  \draw (1.2,-0.2) node {$h_{i,j}$};
  \node at (0.5,-0.6) {$R^{\NE}_{i,j}$};
\end{tikzpicture}.
$$
in an $n$-hive $h$. Then we define $$R^{\NE}_{i,j}(h):=h_{i,j}+h_{i-1,j-1}-h_{i,j-1}-h_{i-1,j}= (h_{i,j}-h_{i,j-1})-(h_{i-1,j}-h_{i-1,j-1}),$$
and by definition $R^{\NE}_{i,j}(h) \geq 0$.

Fix $\lambda,\mu,\nu \in \Par_n $ with $|\lambda| + |\mu| = |\nu|$, and consider a subset $F \subseteq \{(i,j): 1 \leq j < i \leq n \} $. We then define the {\em Hive Kogan face} \cite[\S5.3]{Mrigendra:saturation} of $ \VHive(\lambda,\mu,\nu)$ associated to $F$ as follows
\begin{equation}
    \label{eq:hive-kogan}
\VHive(\lambda,\mu,\nu,F):=\{ h \in \VHive(\lambda,\mu,\nu): R^{\NE}_{i,j}(h)=0 \quad \forall (i,j) \in F \}.
\end{equation}
For $w \in S_n,$ we define
$$\VHive(\lambda,\mu,\nu,w):= \bigcup_{\substack{F: F \text{ is reduced},\\ \hat{w}_F=w}} \VHive(\lambda,\mu,\nu,F),$$
where $\hat{w}_F$ is defined in \ref{sec:Kogan}.
\begin{theorem}
\begin{upshape}
 \cite[Theorem 5.4]{Mrigendra:saturation}   
\end{upshape}
 Let $w \in S_n$ and $\lambda,\mu,\nu \in \Par_n$ such that $|\lambda|+|\mu|=|\nu|$. Then $c_{\lambda,\mu}^{\nu}(w)$ is the cardinality of $\VHive_{\mathbb{Z}}(\lambda,\mu,\nu,ww_0),$ where $w_0$ is the longest permutation in $S_n$.   
\end{theorem}
An {\em edge-labeled $n$-hive} is a graph in the shape of an equilateral triangle whose edges are labeled with real numbers (see Figure~\ref{fig:hive-edge}), satisfying the following conditions:
\begin{itemize}
    \item \textbf{Rhombus condition:} For each rhombus in Figure~\ref{fig:Rhombus}, the edge labels satisfy $p \geq r $ and $ q \geq s$.
    \item \textbf{Triangle condition:} In every triangle, the sum of the labels on the two oblique sides equals the label on the horizontal side.
\end{itemize}
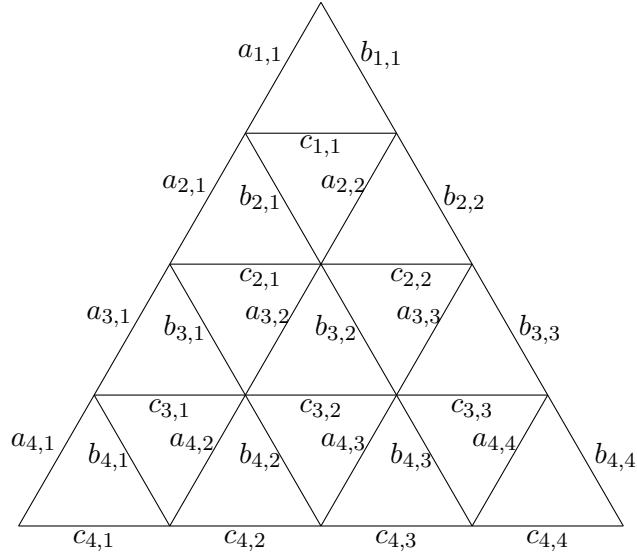
\begin{figure}
    \centering
\begin{tikzpicture}[scale=2]
    \node at (1.5+0.1,1.8*1.732) {$a_{1,1}$};
    \node at (1+0.1,1.3*1.732) {$a_{2,1}$};
    \node at (1+1.15,1.3*1.732) {$a_{2,2}$};
    \node at (0.5+0.1,0.8*1.732) {$a_{3,1}$};
    \node at (0.5+1.15,0.8*1.732) {$a_{3,2}$};
    \node at (0.5+1.15+1,0.8*1.732) {$a_{3,3}$};
    \node at (0+0.1,{(0.4*1.732)-0.15}) {$a_{4,1}$};
    \node at (0+1.15,{(0.4*1.732)-0.15}) {$a_{4,2}$};
    \node at (0+1.15+1,{(0.4*1.732)-0.15}) {$a_{4,3}$};
    \node at (0+1.15+1+1,{(0.4*1.732)-0.15}) {$a_{4,4}$};
    
    \node at (2,1.45*1.732) {$c_{1,1}$};
    \node at (1.6,1.732-0.1) {$c_{2,1}$};
    \node at (1.6+1,1.732-0.1) {$c_{2,2}$};
    \node at (1,0.5*1.732-0.1) {$c_{3,1}$};
    \node at (1+1,0.5*1.732-0.1) {$c_{3,2}$};
    \node at (1+2,0.5*1.732-0.1) {$c_{3,3}$};
    \node at (0.5,-0.1) {$c_{4,1}$};
    \node at (0.5+1,-0.1) {$c_{4,2}$};
    \node at (0.5+2,-0.1) {$c_{4,3}$};
    \node at (0.5+3,-0.1) {$c_{4,4}$};

    \node at (2.5-0.1,1.8*1.732) {$b_{1,1}$};
    \node at (3-0.05,1.25*1.732) {$b_{2,2}$};
    \node at (3-1.4,1.25*1.732) {$b_{2,1}$};
    \node at (3.5-0.05,0.75*1.732) {$b_{3,3}$};
    \node at (3.5-1.4,0.75*1.732) {$b_{3,2}$};
    \node at (3.5-2.4,0.75*1.732) {$b_{3,1}$};
    \node at (4.1-0.15,{0.35*1.732-0.15}) {$b_{4,4}$};
    \node at (4.1-1.5,{0.35*1.732-0.15}) {$b_{4,3}$};
    \node at (4.1-2.5,{0.35*1.732-0.15}) {$b_{4,2}$};
    \node at (4.1-3.5,{0.35*1.732-0.15}) {$b_{4,1}$};

    \draw (0,0) -- (4,0);
    \draw (0.5, 0.5*1.732) -- (3.5, 0.5*1.732);
    \draw (1, 1.732) -- (2, 1.732) -- (3, 1.732);
    \draw (1.5, 1.5*1.732) -- (2.5, 1.5*1.732);
    \draw (0,0) -- (0.5, 0.5*1.732) -- (1, 1.732) -- (1.5, 1.5*1.732) -- (2, 2*1.732);
    \draw (1,0) -- (1.5, 0.5* 1.732) -- (2, 1.732) -- (2.5, 1.5*1.732);
    \draw (2,0) -- (2.5, 0.5* 1.732) -- (3, 1.732);
    \draw (3,0) -- (3.5, 0.5* 1.732);
    \draw (1,0) -- (0.5, 0.5*1.732);
    \draw (2,0) -- (1.5, 0.5*1.732) -- (1,1.732);
    \draw (3,0) -- (2.5, 0.5*1.732);
    \draw (2.5, 0.5*1.732) -- (2, 1.732) -- (1.5, 1.5*1.732);
    \draw (4,0) -- (3.5, 0.5*1.732);
    \draw (3.5, 0.5*1.732) -- (3, 1.732) -- (2.5, 1.5*1.732) -- (2, 2*1.732);
\end{tikzpicture}
    \caption{An edge-labeled $4$-hive}
    \label{fig:hive-edge}
\end{figure}
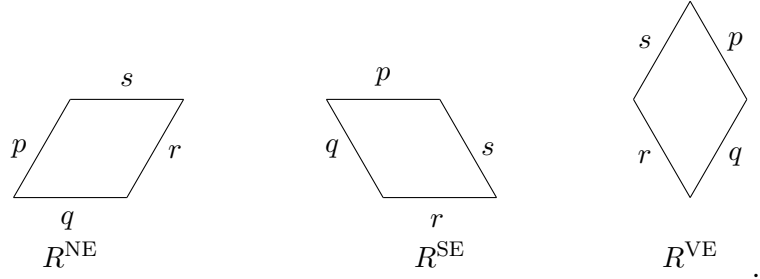
\begin{figure}
    \centering
\begin{tikzpicture}[scale=1.5]
  \draw (0,0) -- (1,0);
  \draw (0,0) -- (0.5, 0.5*1.732);
  \draw (0.5, 0.5*1.732) -- (1.5, 0.5*1.732);
  \draw (1,0) -- (1.5,0.5*1.732);
  \node at (0.5,-0.5) {$R^{\NE}$};
  \node at (0.15-0.1,0.25*1.732) {$p$};
  \node at (1.325+0.1,0.408) {$r$};
  \node at (0.475,-0.1-0.1) {$q$};
  \node at (1,0.941+0.1) {$s$};
\end{tikzpicture}
\quad \quad \quad \quad 
\begin{tikzpicture}[scale=1.5]
  \draw (0,0) -- (1,0);
  \draw (0,0) -- (-0.5, 0.5*1.732);
  \draw (-0.5, 0.5*1.732) -- (0.5, 0.5*1.732);
  \draw (1,0) -- (0.5,0.5*1.732);
  \node at (0.5,-0.5) {$R^{\SE}$};
  \node at (0,0.5*1.732+0.1+0.1) {$p$};
  \node at (-0.3-0.15,0.25*1.732) {$q$};
  \node at (0.475,-0.1-0.1) {$r$};
  \node at (0.825+0.1,0.25*1.732) {$s$};
\end{tikzpicture}
\quad \quad \quad \quad 
\begin{tikzpicture}[scale=1.5]
  \draw (0,0) -- (0.5,0.5*1.732);
  \draw (1,0) -- (0.5,0.5*1.732);
  \draw (0.5,-0.5*1.732) -- (0,0);
  \draw (0.5,-0.5*1.732) -- (1,0);
  \node at (0.5,-0.5*1.732-0.5) {$R^{\VE}$};
  \node at (0.8+0.1,0.25*1.732+0.15/2) {$p$};
  \node at (0.8+0.1,-0.25*1.732-0.15/2) {$q$};
  \node at (0.2-0.1,-0.25*1.732-0.15/2) {$r$};
  \node at (0.2-0.1,0.25*1.732+0.15/2) {$s$};
\end{tikzpicture}.
    \caption{Rhombus inequalities}
    \label{fig:Rhombus}
\end{figure} 
We can easily switch from a vertex-labeled $n$-hive to an edge-labeled $n$-hive as follows. Let $h=(h_{i,j})_{ 0 \leq j \leq i \leq n }$ be a vertex-labeled $n$-hive. Then the corresponding edge-labeled $n$-hive $h(a,b,c)$ is given by:
\begin{equation}
    \label{eq:ver-to-edge}
    a_{i,j}=h_{i,j-1}-h_{i-1,j-1}, b_{i,j}= h_{i,j}-h_{i-1,j-1}, c_{i,j}=h_{i,j}-h_{i,j-1}.
\end{equation}
Conversely, we can switch from an edge-labeled $n$-hive to a vertex-labeled $n$-hive using \eqref{eq:ver-to-edge}. Fix $\lambda,\mu,\nu \in \Par_n $ with $|\lambda| + |\mu| = |\nu|$. We then define $\EHive_{\mathbb{Z}}(\lambda,\mu,\nu)$ to be the set of all integer edge-labeled hives $h(a,b,c)$ satisfying the boundary conditions $a_{i,1}=\lambda_i$, $b_{i,i}=\nu_i$, and $c_{n,i}=\mu_i$ for $1 \leq i \leq n$. This yields the following bijection.
\begin{proposition}
\label{prop:ve}
Fix $\lambda,\mu,\nu \in \Par_n$ with $|\lambda| + |\mu| = |\nu|$. Then the map $$\Gamma: \VHive_{\mathbb{Z}}(\lambda,\mu,\nu) \rightarrow \EHive_{\mathbb{Z}}(\lambda,\mu,\nu), h \mapsto h(a,b,c)$$ is a bijection. 
\end{proposition}
\section{proof of the main theorem}
The following result is the main theorem of this article.
\begin{theorem}
\label{thm:main}
    Let $\lambda,\mu,\nu \in \Par_n $ such that $|\lambda|+|\mu|=|\nu|$. Assume there exist indices $i,j,k$ such that $i+j+k=n$ and $\lambda_i+\mu_j>\nu_1 +\nu_{n-k+1}$. Suppose, moreover, that $\lambda_i >\lambda_{i+1}, \mu_i >\mu_{i+1}$ and $\nu_{n-k}>\nu_{n-k+1}$.    Then, we have, for 
    $w \in S_n$,
    $$c^{\nu}_{\lambda,\mu}(w)=c^{\nu -(1^{n-k})}_{\lambda-(1^i),\mu-(1^j)}(w).$$
\end{theorem}
\begin{proof}
Let $T \in \SSYT_{\lambda,\mu}^{\nu}$. Then assume that $\Upsilon^{-1}(T)=(t_{i,j})_{1 \leq j \leq i \leq n}$. Then, by \cite[Proposition 5.2]{Mrigendra:saturation}, we can say that the hive $h \in \VHive_{\mathbb{Z}}(\lambda,\mu,\nu)$ which corresponds to $T$ is given by
\begin{equation}
    \label{eq:}
    h_{i,j}=\sum_{k=1}^{i}\lambda_k + \sum_{k=1}^{j}t_{i,k},
\end{equation}
for $0 \leq j \leq i \leq n$. We note that (see \S\ref{sec:hive-kogan})
\begin{equation}
    \label{eq:rhombus}
  R^{\NE}_{i,j}(h)= (h_{i,j}-h_{i,j-1})-(h_{i-1,j}-h_{i-1,j-1})=t_{i,j}-t_{i-1,j}=\NE_{i,j}(\Upsilon^{-1}(T)).  
\end{equation}
Since $i+j+k=n$ and $ c^{(1^{n-k})}_{(1^{i}),(1^{j})}=1,$ then the unique semi-standard Young tableau of shape $(1^{j})$ which is $(1^{i})$-dominant and weight is $(1^{n-k})-(1^{i})$ is given below
$$
\Bar{T}=
\ytableausetup{mathmode,
notabloids,boxsize=2.5em}
  \begin{ytableau}
    i+1 \\ i+2 \\ \dots \\ n-k  
  \end{ytableau}.
$$
Consider $i=2,j=3,k=1$. Then $n=i+j+k=6$. Then the GT pattern $\Upsilon^{-1}(\Bar{T})$ of size $n$ is given below 
$$
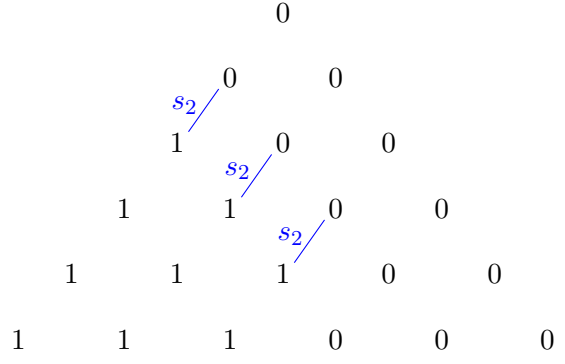
\begin{tikzpicture}[scale=1]
    \draw (2.8+0.7,2.5*1.732) node {$0$};
    
    \draw (2.8,2*1.732) node {$0$};
    \draw (2.8+1.4,2*1.732) node {$0$};
    
	\draw (2.1,1.5*1.732) node {$1$};
    \draw (3.5,1.5*1.732) node {$0$};
    \draw (3.5+1.4,1.5*1.732) node {$0$};
    
	\draw (1.4,1.732) node {$1$};
	\draw (2.8,1.732) node {$1$};
    \draw (4.2,1.732) node {$0$};
    \draw (4.2+1.4,1.732) node {$0$};
    
	\draw (0.7,0.5*1.732) node {$1$};
	\draw (2.1,0.5*1.732) node {$1$};
	\draw (3.5,0.5*1.732) node {$1$};
    \draw (4.9,0.5*1.732) node {$0$};
    \draw (4.9+1.4,0.5*1.732) node {$0$};
    
	\draw (0,0) node {$1$};
	\draw (1.4,0) node {$1$};
	\draw (2.8,0) node {$1$};
	\draw (4.2,0) node {$0$};
    \draw (5.6,0) node {$0$};
    \draw (5.6+1.4,0) node {$0$};
    
    \draw[blue] (2.1+0.15,1.5*1.732+0.15) -- (2.8-0.15,2*1.732-0.15);
    \draw[blue] (2.8+0.15,1.732+0.15) -- (3.5-0.15,1.5*1.732-0.15);
    \draw[blue] (3.5+0.15,0.5*1.732+0.15) -- (4.2-0.15,1.732-0.15);
    \draw[blue] (2.1+0.1,1.5*1.732+0.5) node {$s_2$};
    \draw[blue] (2.8+0.1,1.732+0.5) node {$s_2$};
    \draw[blue] (3.5+0.1,0.5*1.732+0.5) node {$s_2$};
\end{tikzpicture}.
$$
There exists a unique $\Bar{h} \in \VHive_{\mathbb{Z}}((1^{i}),(1^{j}),(1^{n-k}))$. Let $h \in \VHive_{\mathbb{Z}}(\lambda,\mu,\nu,ww_0)$. We now define an array $h-\Bar{h}$ of size $n$ by
$$(h-\Bar{h})_{i,j}=h_{i,j}-\Bar{h}_{i,j},$$ for $0 \leq j \leq i \leq n.$ Then we show that $h-\Bar{h} \in \VHive_{\mathbb{Z}}(\lambda-(1^i),\mu-(1^j),\nu-(1^{n-k}),ww_0)$. Now $\Gamma(h) \in \EHive_{\mathbb{Z}}(\lambda,\mu,\nu), \Gamma(\Bar{h}) \in \EHive_{\mathbb{Z}}((1^i),(1^j),(1^{n-k}))$. Then, under the hypotheses in Theorem~\ref{thm:main}, it follows from that \cite[Theorem 3.6]{Hive-reduction} the following map is a bijection.  
$$\Psi_{i,j,k}: \EHive_{\mathbb{Z}}(\lambda,\mu,\nu) \rightarrow \EHive_{\mathbb{Z}}(\lambda-(1^i),\mu-(1^j),\nu-(1^{n-k})), \, \mathfrak{h}\mapsto \mathfrak{h}-\Gamma(\Bar{h}),$$
where each edge label of the difference hive is the difference of the corresponding labels in the original two $n$-hives.

Thus we obtain $\Psi_{i,j,k}(\Gamma(h))=\Gamma(h)-\Gamma(\Bar{h}) \in \EHive_{\mathbb{Z}}(\lambda-(1^i),\mu-(1^j),\nu-(1^{n-k}))$. Since $\Gamma^{-1}(\Gamma(h)-\Gamma (\Bar{h})) =h-\Bar{h},$ $h-\Bar{h} \in \VHive_{\mathbb{Z}}(\lambda-(1^i),\mu-(1^j),\nu-(1^{n-k}))$.

Now we obtain
\begin{equation}
\label{eq:1}
\NE_{k,l}(\Upsilon^{-1}(\Bar{T})) = \left\{ 
    \begin{array}{ll}
        1 & \quad \text{if} \quad (k,l) \in \{ (i+1,1),(i+2,2),\dots, (i+j,j) \},  \\
        0 & \quad  \text{otherwise}. 
    \end{array}  \right. 
\end{equation}
Therefore, using \eqref{eq:rhombus}, we get
\begin{equation}
\label{eq:2}
R^{\NE}_{k,l}(\Bar{h}) = \left\{ 
    \begin{array}{ll}
        1 & \quad \text{if} \quad (k,l) \in \{ (i+1,1),(i+2,2),\dots, (i+j,j) \},  \\
        0 & \quad  \text{otherwise}. 
    \end{array}  \right. 
\end{equation}
Since $R^{\NE}_{k,l}(h-\Bar{h})=R^{\NE}_{k,l}(h)-R^{\NE}_{k,l}(\Bar{h}),$ we have
\begin{equation}
\label{eq:3}
R^{\NE}_{k,l}(h-\Bar{h}) = \left\{ 
    \begin{array}{ll}
        R^{\NE}_{k,l}(h)-1 & \quad \text{if} \quad (k,l) \in \{ (i+1,1),(i+2,2),\dots, (i+j,j) \},  \\
        R^{\NE}_{k,l}(h) & \quad \text{otherwise}. 
    \end{array}  \right. 
\end{equation}
Since $h-\Bar{h} \in \VHive_{\mathbb{Z}}(\lambda-(1^i),\mu-(1^j),\nu-(1^{n-k}))$, it follows directly from the definition that $R^{\NE}_{k,l}(h-\Bar{h}) \geq 0$ for each $k,l$. Consequently, \eqref{eq:3} implies that
\begin{equation}
\label{eq:4}
R^{\NE}_{k,l}(h) \geq \left\{ 
    \begin{array}{ll}
        1 & \quad \text{if} \quad (k,l) \in \{ (i+1,1),(i+2,2),\dots, (i+j,j) \},  \\
        0 & \quad \text{otherwise}. 
    \end{array}  \right. 
\end{equation}
Using \eqref{eq:rhombus}, we obtain $R^{\NE}_{k,l}(h) =\NE_{k,l}(\Upsilon^{-1}(T))$. Therefore,
\begin{equation}
\label{eq:NE}
\NE_{k,l}(\Upsilon^{-1}(T)) \geq \left\{ 
    \begin{array}{ll}
        1 & \quad \text{if} \quad (k,l) \in \{ (i+1,1),(i+2,2),\dots, (i+j,j) \},  \\
        0 & \quad \text{otherwise}. 
    \end{array}  \right. 
\end{equation}
Recall from \S\ref{sec:hive} that $c_{\lambda,\mu}^{\nu}(w)$ is the cardinality of
\begin{equation}
    \label{eq:a}
    \VHive_{\mathbb{Z}}(\lambda,\mu,\nu,ww_0)= \bigcup_{\substack{F: F \text{ is reduced},\\ \hat{w}_F=ww_0}} \VHive_{\mathbb{Z}}(\lambda,\mu,\nu,F),
\end{equation}
where 
\begin{equation}
    \label{eq:b}
  \VHive_{\mathbb{Z}}(\lambda,\mu,\nu,F):=\{ h \in \VHive_{\mathbb{Z}}(\lambda,\mu,\nu): R^{\NE}_{i,j}(h)=0 \quad \forall (i,j) \in F \}.  
\end{equation}
Now let $h \in \VHive_{\mathbb{Z}}(\lambda,\mu,\nu,ww_0)$. Then
\begin{equation}
    \label{eq:p}
    h \in \VHive_{\mathbb{Z}}(\lambda,\mu,\nu,\mathcal{F}),
\end{equation}
for some subset $\mathcal{F} \subseteq \{ (i,j):1 \leq j < i \leq n \}$ such that $\mathcal{F}$ is reduced, $\hat{w}_{\mathcal{F}}=ww_0$. Now using \eqref{eq:rhombus}, \eqref{eq:4},\eqref{eq:NE}, we have 
$$R^{\NE}_{k,l}(h)= \NE_{k,l}(\Upsilon^{-1}(T)) \geq 1 \text{ for } (k,l) \in \{ (i+1,1),(i+2,2),\dots, (i+j,j) \},$$
which implies 
\begin{equation}
    \label{eq:F}
    \mathcal{F} \cap \{(i+1,1),(i+2,2),\dots, (i+j,j) \} =\emptyset.
\end{equation}
Also, by \eqref{eq:2}, \eqref{eq:F}, we get
\begin{equation}
    \label{eq:q}
    \Bar{h} \in \VHive_{\mathbb{Z}}((1^i),(1^j),(1^{n-k}),\mathcal{F}).
\end{equation}
Using \eqref{eq:p}, \eqref{eq:q} and the equality $ R^{\NE}_{k,l}(h-\Bar{h})=R^{\NE}_{k,l}(h)-R^{\NE}_{k,l}(\Bar{h})$, we obtain
$$h-\Bar{h} \in \VHive_{\mathbb{Z}}(\lambda-(1^i),\mu-(1^j),\nu-(1^{n-k}),\mathcal{F}).$$
Hence, by \eqref{eq:a}, we have 
$$ h-\Bar{h} \in \VHive_{\mathbb{Z}}(\lambda-(1^i),\mu-(1^j),\nu-(1^{n-k}),ww_0).$$
Therefore, under the hypotheses in Theorem~\ref{thm:main}, the map below $$\Gamma^{-1} \circ \Psi_{i,j,k} \circ \Gamma : \VHive_{\mathbb{Z}}(\lambda,\mu,\nu,ww_0) \rightarrow  \VHive_{\mathbb{Z}}(\lambda-(1^i),\mu-(1^j),\nu-(1^{n-k}),ww_0), \, h \mapsto h-\Bar{h}$$ is a bijection.
\end{proof}
\begin{corollary}
    We can deduce Theorem~\ref{thm:second}, i.e., the second reduction formula for $c^{\nu}_{\lambda,\mu}$ from the above Theorem by setting $w=w_0$.
\end{corollary}
\begin{example}
Let $i=j=k=1$ so that $n=3$. Also, let $\lambda=(3,1,0),\mu=(3,1,0),\nu=(4,3,1) \in \Par_3$. Then it is clear that $\lambda_1 > \lambda_2, \mu_1 > \mu_2, \nu_2 > \nu_3$ and $\lambda_1+\mu_1 > \nu_1 +\nu_3$.

We now recall from Corollary~\ref{cor:LR-rule} that $c^{\nu}_{\lambda,\mu}$ is the cardinality of the set $\SSYT_{\lambda,\mu}^{\nu}, $ where \[ \SSYT_{\lambda,\mu}^{\nu} =\{T \in \SSYT(\mu): r_{T^{\lambda}}*r_T \text{ is a dominant word of weight } \nu\}.\]
For the given $\lambda,\mu,\nu$, it can be checked that 
$$\SSYT_{\lambda,\mu}^{\nu}=\Bigl\{  \ytableausetup{mathmode,
notabloids,boxsize=1.5em}
  \begin{ytableau}
    1&2&2\\3  
  \end{ytableau},
  \begin{ytableau}
    1&2&3\\2  
  \end{ytableau}
\Bigl\}
\text{ and } 
\SSYT_{\lambda-(1),\mu-(1)}^{\nu-(1,1)}=\Bigl\{  \ytableausetup{mathmode,
notabloids,boxsize=1.5em}
  \begin{ytableau}
    1&2\\3  
  \end{ytableau},
  \begin{ytableau}
    1&3\\2  
  \end{ytableau}
  \Bigl\}.$$
Thus we have $$c^{\nu}_{\lambda,\mu}=c^{\nu -(1^{n-k})}_{\lambda-(1^i),\mu-(1^j)}=2.$$
Now, let $w_1=s_1s_2$.
Then the crystal graph of $\mathfrak{B}_{s_1s_2}(3,1)$ is 
$$
\begin{tikzpicture}[scale=0.8]
    \draw (0,0)--(0,1.4)--(2.1,1.4)--(2.1,0.7)--(0.7,0.7)--(0.7,0)--(0,0);
    \draw (0,0.7)--(0.7,0.7)--(0.7,1.4);
    \draw (1.4,0.7)--(1.4,1.4);
    \node at (0.35,0.35) {$2$};
    \node at (0.35,0.35+0.7) {$1$};
    \node at (0.35+0.7,0.35+0.7) {$1$};
    \node at (0.35+0.7+0.7,0.35+0.7) {$1$};
    \draw[->] (2.4,0.7) -- (3.6,0.7);
    \node at (2.9,1) {$2$};

    \draw (0+3.9,0)--(0+3.9,1.4)--(2.1+3.9,1.4)--(2.1+3.9,0.7)--(0.7+3.9,0.7)--(0.7+3.9,0)--(0+3.8,0);
    \draw (0+3.9,0.7)--(0.7+3.9,0.7)--(0.7+3.9,1.4);
    \draw (1.4+3.9,0.7)--(1.4+3.9,1.4);
    \node at (0.35+3.9,0.35) {$3$};
    \node at (0.35+3.9,0.35+0.7) {$1$};
    \node at (0.35+0.7+3.9,0.35+0.7) {$1$};
    \node at (0.35+0.7+0.7+3.9,0.35+0.7) {$1$};
    \draw[->] (2.4+3.9,0.7) -- (3.6+3.9,0.7);
    \node at (2.9+3.9,1) {$1$};

    \draw (0+7.8,0)--(0+7.8,1.4)--(2.1+7.8,1.4)--(2.1+7.8,0.7)--(0.7+7.8,0.7)--(0.7+7.8,0)--(0+7.8,0);
    \draw (0+7.8,0.7)--(0.7+7.8,0.7)--(0.7+7.8,1.4);
    \draw (1.4+7.8,0.7)--(1.4+7.8,1.4);
    \node at (0.35+7.8,0.35) {$3$};
    \node at (0.35+7.8,0.35+0.7) {$1$};
    \node at (0.35+0.7+7.8,0.35+0.7) {$1$};
    \node at (0.35+0.7+0.7+7.8,0.35+0.7) {$2$};
    \draw[->] (2.4+7.8,0.7) -- (3.6+7.8,0.7);
    \node at (2.9+7.8,1) {$1$};

    \draw (0+11.7,0)--(0+11.7,1.4)--(2.1+11.7,1.4)--(2.1+11.7,0.7)--(0.7+11.7,0.7)--(0.7+11.7,0)--(0+11.7,0);
    \draw (0+11.7,0.7)--(0.7+11.7,0.7)--(0.7+11.7,1.4);
    \draw (1.4+11.7,0.7)--(1.4+11.7,1.4);
    \node at (0.35+11.7,0.35) {$3$};
    \node at (0.35+11.7,0.35+0.7) {$1$};
    \node at (0.35+0.7+11.7,0.35+0.7) {$2$};
    \node at (0.35+0.7+0.7+11.7,0.35+0.7) {$2$};
    \draw[->] (2.4+11.7,0.7) -- (3.6+11.7,0.7);
    \node at (2.9+11.7,1) {$1$};

    \draw (0+15.6,0)--(0+15.6,1.4)--(2.1+15.6,1.4)--(2.1+15.6,0.7)--(0.7+15.6,0.7)--(0.7+15.6,0)--(0+15.6,0);
    \draw (0+15.6,0.7)--(0.7+15.6,0.7)--(0.7+15.6,1.4);
    \draw (1.4+15.6,0.7)--(1.4+15.6,1.4);
    \node at (0.35+15.6,0.35) {$3$};
    \node at (0.35+15.6,0.35+0.7) {$2$};
    \node at (0.35+0.7+15.6,0.35+0.7) {$2$};
    \node at (0.35+0.7+0.7+15.6,0.35+0.7) {$2$};
   \end{tikzpicture},
$$
whereas that of $\mathfrak{B}_{s_1s_2}(2,1)$ is
$$
\begin{tikzpicture}[scale=0.8]
    \draw (0,0)--(0,1.4)--(1.4,1.4)--(1.4,0.7)--(0.7,0.7)--(0.7,0)--(0,0);
    \draw (0,0.7)--(0.7,0.7)--(0.7,1.4);
    \node at (0.35,0.35) {$2$};
    \node at (0.35,0.35+0.7) {$1$};
    \node at (0.35+0.7,0.35+0.7) {$1$};
    \draw[->] (1.7,0.7) -- (2.7,0.7);
    \node at (2.2,1) {$2$};

    \draw (0+3,0)--(0+3,1.4)--(1.4+3,1.4)--(1.4+3,0.7)--(0.7+3,0.7)--(0.7+3,0)--(0+3,0);
    \draw (0+3,0.7)--(0.7+3,0.7)--(0.7+3,1.4);
    \node at (0.35+3,0.35) {$3$};
    \node at (0.35+3,0.35+0.7) {$1$};
    \node at (0.35+0.7+3,0.35+0.7) {$1$};
    \draw[->] (1.7+3,0.7) -- (2.7+3,0.7);
    \node at (2.2+3,1) {$1$};

    \draw (0+6,0)--(0+6,1.4)--(1.4+6,1.4)--(1.4+6,0.7)--(0.7+6,0.7)--(0.7+6,0)--(0+6,0);
    \draw (0+6,0.7)--(0.7+6,0.7)--(0.7+6,1.4);
    \node at (0.35+6,0.35) {$3$};
    \node at (0.35+6,0.35+0.7) {$1$};
    \node at (0.35+0.7+6,0.35+0.7) {$2$};
    \draw[->] (1.7+6,0.7) -- (2.7+6,0.7);
    \node at (2.2+6,1) {$1$};

    \draw (0+9,0)--(0+9,1.4)--(1.4+9,1.4)--(1.4+9,0.7)--(0.7+9,0.7)--(0.7+9,0)--(0+9,0);
    \draw (0+9,0.7)--(0.7+9,0.7)--(0.7+9,1.4);
    \node at (0.35+9,0.35) {$3$};
    \node at (0.35+9,0.35+0.7) {$2$};
    \node at (0.35+0.7+9,0.35+0.7) {$2$};
\end{tikzpicture}.
$$
As previously stated in Theorem~\ref{thm:joseph}, $c^{\nu}_{\lambda,\mu}(w)$ is the cardinality of the set
\begin{equation}
    \label{eq:refined-LR}
    \SSYT_{\lambda,\mu}^{\nu}(w) =\{T \in \mathfrak{B}_w(\mu): r_{T^{\lambda}}*r_T \text{ is a dominant word of weight } \nu\}.
\end{equation} 
Then, using \eqref{eq:refined-LR}, we get
$$\SSYT_{\lambda,\mu}^{\nu}(w)=\Bigl\{  \ytableausetup{mathmode,
notabloids,boxsize=1.5em}
  \begin{ytableau}
    1&2&2\\3  
  \end{ytableau} 
  \Bigl\} \quad 
\text{and} \quad  
\SSYT_{\lambda-(1),\mu-(1)}^{\nu-(1,1)}(w)=\Bigl\{  \ytableausetup{mathmode,
notabloids,boxsize=1.5em}
  \begin{ytableau}
    1&2\\3  
  \end{ytableau} \Bigl\}.$$
Consequently, we find that
$$c^{\nu}_{\lambda,\mu}(w)=c^{\nu-(1,1)}_{\lambda-(1),\mu-(1)}(w)=1.$$
Next, we take $w'=s_2s_1$. Then the crystal graph of $\mathfrak{B}_{s_2s_1}(3,1)$ is 
$$
\begin{tikzpicture}[scale=0.65]
    \draw (0,0)--(0,1.4)--(2.1,1.4)--(2.1,0.7)--(0.7,0.7)--(0.7,0)--(0,0);
    \draw (0,0.7)--(0.7,0.7)--(0.7,1.4);
    \draw (1.4,0.7)--(1.4,1.4);
    \node at (0.35,0.35) {$2$};
    \node at (0.35,0.35+0.7) {$1$};
    \node at (0.35+0.7,0.35+0.7) {$1$};
    \node at (0.35+0.7+0.7,0.35+0.7) {$1$};
    \draw[->] (2.4,0.7) -- (3.6,0.7);
    \node at (2.9,1) {$1$};

    \draw (0+3.9,0)--(0+3.9,1.4)--(2.1+3.9,1.4)--(2.1+3.9,0.7)--(0.7+3.9,0.7)--(0.7+3.9,0)--(0+3.8,0);
    \draw (0+3.9,0.7)--(0.7+3.9,0.7)--(0.7+3.9,1.4);
    \draw (1.4+3.9,0.7)--(1.4+3.9,1.4);
    \node at (0.35+3.9,0.35) {$2$};
    \node at (0.35+3.9,0.35+0.7) {$1$};
    \node at (0.35+0.7+3.9,0.35+0.7) {$1$};
    \node at (0.35+0.7+0.7+3.9,0.35+0.7) {$2$};
    \draw[->] (2.4+3.9,0.7) -- (3.6+3.9,0.7);
    \node at (2.9+3.9,1) {$1$};

    \draw (0+7.8,0)--(0+7.8,1.4)--(2.1+7.8,1.4)--(2.1+7.8,0.7)--(0.7+7.8,0.7)--(0.7+7.8,0)--(0+7.8,0);
    \draw (0+7.8,0.7)--(0.7+7.8,0.7)--(0.7+7.8,1.4);
    \draw (1.4+7.8,0.7)--(1.4+7.8,1.4);
    \node at (0.35+7.8,0.35) {$2$};
    \node at (0.35+7.8,0.35+0.7) {$1$};
    \node at (0.35+0.7+7.8,0.35+0.7) {$2$};
    \node at (0.35+0.7+0.7+7.8,0.35+0.7) {$2$};
    \draw[->] (2.4+7.8,0.7) -- (3.6+7.8,0.7);
    \node at (2.9+7.8,1) {$2$};

    \draw (0+11.7,0)--(0+11.7,1.4)--(2.1+11.7,1.4)--(2.1+11.7,0.7)--(0.7+11.7,0.7)--(0.7+11.7,0)--(0+11.7,0);
    \draw (0+11.7,0.7)--(0.7+11.7,0.7)--(0.7+11.7,1.4);
    \draw (1.4+11.7,0.7)--(1.4+11.7,1.4);
    \node at (0.35+11.7,0.35) {$2$};
    \node at (0.35+11.7,0.35+0.7) {$1$};
    \node at (0.35+0.7+11.7,0.35+0.7) {$2$};
    \node at (0.35+0.7+0.7+11.7,0.35+0.7) {$3$};
    \draw[->] (2.4+11.7,0.7) -- (3.6+11.7,0.7);
    \node at (2.9+11.7,1) {$2$};

    \draw (0+15.6,0)--(0+15.6,1.4)--(2.1+15.6,1.4)--(2.1+15.6,0.7)--(0.7+15.6,0.7)--(0.7+15.6,0)--(0+15.6,0);
    \draw (0+15.6,0.7)--(0.7+15.6,0.7)--(0.7+15.6,1.4);
    \draw (1.4+15.6,0.7)--(1.4+15.6,1.4);
    \node at (0.35+15.6,0.35) {$2$};
    \node at (0.35+15.6,0.35+0.7) {$1$};
    \node at (0.35+0.7+15.6,0.35+0.7) {$3$};
    \node at (0.35+0.7+0.7+15.6,0.35+0.7) {$3$};
    \draw[->] (2.4+15.6,0.7) -- (3.6+15.6,0.7);
    \node at (2.9+15.6,1) {$2$};

    \draw (0+19.5,0)--(0+19.5,1.4)--(2.1+19.5,1.4)--(2.1+19.5,0.7)--(0.7+19.5,0.7)--(0.7+19.5,0)--(0+19.5,0);
    \draw (0+19.5,0.7)--(0.7+19.5,0.7)--(0.7+19.5,1.4);
    \draw (1.4+19.5,0.7)--(1.4+19.5,1.4);
    \node at (0.35+19.5,0.35) {$3$};
    \node at (0.35+19.5,0.35+0.7) {$1$};
    \node at (0.35+0.7+19.5,0.35+0.7) {$3$};
    \node at (0.35+0.7+0.7+19.5,0.35+0.7) {$3$};
\end{tikzpicture},   
$$
while that of $\mathfrak{B}_{s_2s_1}(2,1)$ is
$$
\begin{tikzpicture}[scale=0.8]
    \draw (0,0)--(0,1.4)--(1.4,1.4)--(1.4,0.7)--(0.7,0.7)--(0.7,0)--(0,0);
    \draw (0,0.7)--(0.7,0.7)--(0.7,1.4);
    \node at (0.35,0.35) {$2$};
    \node at (0.35,0.35+0.7) {$1$};
    \node at (0.35+0.7,0.35+0.7) {$1$};
    \draw[->] (1.7,0.7) -- (2.7,0.7);
    \node at (2.2,1) {$1$};

    \draw (0+3,0)--(0+3,1.4)--(1.4+3,1.4)--(1.4+3,0.7)--(0.7+3,0.7)--(0.7+3,0)--(0+3,0);
    \draw (0+3,0.7)--(0.7+3,0.7)--(0.7+3,1.4);
    \node at (0.35+3,0.35) {$2$};
    \node at (0.35+3,0.35+0.7) {$1$};
    \node at (0.35+0.7+3,0.35+0.7) {$2$};
    \draw[->] (1.7+3,0.7) -- (2.7+3,0.7);
    \node at (2.2+3,1) {$2$};

    \draw (0+6,0)--(0+6,1.4)--(1.4+6,1.4)--(1.4+6,0.7)--(0.7+6,0.7)--(0.7+6,0)--(0+6,0);
    \draw (0+6,0.7)--(0.7+6,0.7)--(0.7+6,1.4);
    \node at (0.35+6,0.35) {$2$};
    \node at (0.35+6,0.35+0.7) {$1$};
    \node at (0.35+0.7+6,0.35+0.7) {$3$};
    \draw[->] (1.7+6,0.7) -- (2.7+6,0.7);
    \node at (2.2+6,1) {$2$};

    \draw (0+9,0)--(0+9,1.4)--(1.4+9,1.4)--(1.4+9,0.7)--(0.7+9,0.7)--(0.7+9,0)--(0+9,0);
    \draw (0+9,0.7)--(0.7+9,0.7)--(0.7+9,1.4);
    \node at (0.35+9,0.35) {$3$};
    \node at (0.35+9,0.35+0.7) {$1$};
    \node at (0.35+0.7+9,0.35+0.7) {$3$};
\end{tikzpicture}.
$$
Applying \eqref{eq:refined-LR} then gives
$$\SSYT_{\lambda,\mu}^{\nu}(w')=\Bigl\{  \ytableausetup{mathmode,
notabloids,boxsize=1.5em}
  \begin{ytableau}
    1&2&3\\2  
  \end{ytableau} 
  \Bigl\} \quad 
\text{and}
\quad 
\SSYT_{\lambda-(1),\mu-(1)}^{\nu-(1,1)}(w')=\Bigl\{  \ytableausetup{mathmode,
notabloids,boxsize=1.5em}
  \begin{ytableau}
    1&3\\2  
  \end{ytableau} \Bigl\}.$$
Hence, we obtain
$$c^{\nu}_{\lambda,\mu}(w')= c^{\nu-(1,1)}_{\lambda-(1),\mu-(1)}(w')=1.$$
\end{example} 

\end{document}